\documentclass[12pt,a4paper,reqno]{amsart} 
\usepackage{amssymb}
\usepackage{latexsym}
\usepackage{amsmath}
\usepackage{mathrsfs}
\usepackage[X2,T1]{fontenc}
\usepackage[applemac]{inputenc}
\usepackage{cite}
\usepackage{calc}                   

\newcommand{\scal}[2]{\langle #1,#2\rangle}
\newcommand{\rr}[1]{\mathbf R^{#1}}

\newcommand{\nm}[2]{\Vert #1\Vert _{#2}}

\newcommand{\op}{\operatorname{Op}}

\newcommand{\sets}[2]{\{ \, #1\, ;\, #2\, \} }

\newcommand{\ep}{\varepsilon}

\newcommand{\cdo}{\, \cdot \, }

\newcommand{\loc}{\operatorname{loc}}
\newcommand{\wpr}{{\text{\footnotesize $\#$}}}

\newcommand{\vrum}{\vspace{0.1cm}}

\newcommand{\px}{\psi_{x_0}}
\newcommand{\pxi}{\psi_{\xi_0}}

\newcommand{\nn}[1]{{\mathbf N}^{#1}}
\newcommand{\maclS}{\mathcal S}
\newcommand{\mascF}{\mathscr F}
\newcommand{\mascS}{\mathscr S}
\newcommand{\mascP}{\mathscr P}

\def\cS{{\mathcal S}}

\def\pxi{\langle \xi \rangle}
\def\px{\langle x \rangle}

\numberwithin{equation}{section}          

\newtheorem{thm}{Theorem}
\numberwithin{thm}{section}

\newcommand{\rubrik}{}
\newtheorem{prop}[thm]{Proposition}
\newtheorem{cor}[thm]{Corollary}

\theoremstyle{definition}

\newtheorem{defn}[thm]{Definition}

\theoremstyle{remark}
\newtheorem{rem}[thm]{Remark}

\author{Marco Cappiello}

\address{Dipartimento di Matematica, Universit\`a di Torino, Italy}

\email{marco.cappiello@unito.it}

\author{Joachim Toft}

\address{Department of Mathematics,
Linn{\ae}us University, V{\"a}xj{\"o}, Sweden}

\email{joachim.toft@lnu.se}

\title{Pseudo-differential operators in a Gelfand-Shilov setting}

\begin{document}

\begin{abstract}
We introduce some general classes of pseudodifferential operators with symbols
admitting exponential type growth at infinity and we prove mapping properties for
these operators on Gelfand-Shilov spaces. Moreover, we deduce composition
and certain invariance properties of these classes.
\end{abstract}

\maketitle

\par

\section{Introduction}\label{sec0}

\par

In this paper we study some classes of pseudo-differential operators with ultra-differentiable
symbols $a(x,\xi)$ admitting suitable super-exponential growth at infinity.
Operators of this type are commonly known in the literature as operators of infinite
order and they have been studied in \cite{BM} in the analytic class and
in \cite{Za} in the Gevrey spaces in the case when the symbol has
an exponential growth only with respect to $\xi$. A further step has been
the formulation of a calculus of Fourier integral operators with symbols of
infinite order and its application to hyperbolic Cauchy problems in Gevrey
classes, see \cite{ CZ, CZ2}.

\par

More recently, the first author developed a
parallel global theory by considering symbols which admit exponential
growth also with respect to $x$, namely 
\begin{equation}\label{symbCap}
|\partial_\xi^{\alpha}\partial_x^{\beta}a(x,\xi)| \lesssim \ep^{|\alpha+\beta|}
(\alpha!)^\mu (\beta!)^\nu \pxi^{-|\alpha|}\px^{-|\beta|}e^{h (|x|^{\frac 1s}+|\xi|^{\frac 1s})},
\end{equation}
for every $h >0$ and for some $\ep>0,$ where $\mu\geq1, \nu \geq 1$
and $s \geq \mu+\nu-1$, cf. \cite{Ca1, Ca2}. See also
 \cite{CPP, Pr} for similar results.
 
 \par
 
A natural functional framework for symbols defined by \eqref{symbCap} is given
by the Gelfand-Shilov spaces $\mathcal{S}_s( \rr d)$, introduced in \cite{GS}, and
defined as the spaces of all functions $f \in C^\infty(\rr d)$ which satisfy estimates
of the form
\begin{equation}\label{GSestimate}
\sup_{\alpha, \beta \in \nn d}\sup_{x \in \rr d} \frac{|x^\beta \partial^\alpha f(x)|}
{\ep^{|\alpha+\beta|}(\alpha! \beta!)^s} < \infty
\end{equation}
for some constant $\ep >0$. For $s \geq 1$ they represent a natural global
counterpart of Gevrey classes. However, these spaces are actually defined
also for $\frac 12 \leq s < 1$, hence they are a good functional setting for operators
with symbols admitting a stronger exponential growth than in \eqref{symbCap}.
Together with these spaces, one can also consider their projective version,
namely the space $\Sigma_s(\rr d), s >\frac 12$ which is defined by requiring that
\eqref{GSestimate} holds for every $\ep>0$. 

\par

Nevertheless, the construction of a symbolic calculus in the case
$s < 1$ is a challenging problem, mainly because the elements of
$\mathcal{S}_s( \rr d)$ and $\Sigma_s( \rr d)$ admit entire extensions in the
complex domain satisfying suitable exponential bounds, in this case. See
\cite{GS} for details. 
In particular, for $s<1$, $\mathcal{S}_s( \rr d)$ and $\Sigma_s( \rr d)$ 
lack compactly supported functions which play a fundamental role in the 
construction of a symbol starting from its asymptotic expansion.

\par 

On the other hand, some results in this direction can be proved also in the
quasi-analytic case $s<1$ by using different tools than the usual micro-local
techniques. In this paper we approach the study of pseudo-differential
operators of infinite order on Gelfand-Shilov spaces with a method based
mainly on the use of modulation spaces and of the short time Fourier
transform. In Section \ref{sec1}, after recalling some basic properties of the
spaces $\mathcal{S}_s(\rr d)$ and $\Sigma_s(\rr d)$, we introduce several
general symbol classes. In Section \ref{sec2} we characterize these symbols
in terms of the behavior of their short time Fourier transform. In Section
\ref{sec3} we deduce continuity on $\mathcal{S}
_s(\rr d)$ and $\Sigma_s(\rr d)$, composition and invariance properties for
pseudo-differential operators in our classes.

\par

\section{Preliminaries}\label{sec1}

\par

In this section we recall some basic facts, especially concerning Gel\-fand-Shilov spaces, the
short-time Fourier transform, modulation spaces and pseudo-differential operators.

\par

In the following we shall denote by $\mathscr{S}(\rr d)$ the Schwartz space of rapidly decreasing
functions on $\rr d$ together with their derivatives, and by $\mathscr{S}'(\rr d)$
the corresponding dual space of tempered distributions. 

\subsection{Gelfand-Shilov spaces} 
\par
We start by recalling some facts about Gelfand-Shilov spaces.
Let $0<h,s,t\in \mathbf R$ be fixed. We denote by $\mathcal S_{t,h}^s(\rr d)$
the Banach space of all $f\in C^\infty (\rr d)$ such that
\begin{equation}\label{gfseminorm}
\nm f{\mathcal S_{t,h}^s}\equiv \sup_{\alpha ,\beta \in
\mathbf N^d} \sup_{x \in \rr d} \frac {|x^\beta \partial ^\alpha
f(x)|}{h^{|\alpha | + |\beta |}\alpha !^s\, \beta !^t}<\infty,
\end{equation}
endowed with the norm \eqref{gfseminorm}.

\par

Obviously we have $\mathcal S_{t,h}^s(\rr d)\hookrightarrow \mascS (\rr d)$ for any $h,s,t$
with continuous embedding. 
Furthermore, if $s,t\ge \frac 12$ and $h$
is sufficiently large, then $\mathcal
S_{t,h}^s$ contains all finite linear combinations of Hermite functions.
Since such linear combinations are dense in $\mathscr S (\rr d)$, it follows
that the dual $(\mathcal S_{t,h}^s)'(\rr d)$ of $\mathcal S_{t,h}^s(\rr d)$ is
a Banach space which contains $\mathscr S'(\rr d)$ for such $s$ and $t$.

\par

The \emph{Gelfand-Shilov spaces} $\mathcal S_t^{s}(\rr d)$ and
$\Sigma _t^s(\rr d)$ are defined as the inductive and projective 
limits respectively of $\mathcal S_{t,h}^s(\rr d)$. This implies that
\begin{equation}\label{GSspacecond1}
\mathcal S_t^{s}(\rr d) = \bigcup _{h>0}\mathcal S_{t,h}^s(\rr d)
\quad \text{and}\quad \Sigma _t^{s}(\rr d) =\bigcap _{h>0}
\mathcal S_{t,h}^s(\rr d),
\end{equation}
and that the topology for $\mathcal S_t^{s}(\rr d)$ is the strongest
possible one such that the inclusion map from $\mathcal S_{t,h}^s
(\rr d)$ to $\mathcal S_t^{s}(\rr d)$ is continuous, for every choice 
of $h>0$. The space $\Sigma _s(\rr d)$ is a Fr{\'e}chet space
with seminorms $\nm \cdo {\mathcal S_{t,h}^s}$, $h>0$. Moreover,
$\Sigma _t^s(\rr d)\neq \{ 0\}$, if and only if $s+t\ge 1$ and $(s,t)\neq
(\frac 12,\frac 12)$, and
$\maclS _t^s(\rr d)\neq \{ 0\}$, if and only
if $s+t\ge 1$.

\par

The spaces $\mathcal S_t^{s}(\rr d)$ and $\Sigma _t^s(\rr d)$ can be characterized 
also in terms of the exponential decay of their elements, namely $f \in \mathcal S_t^{s}(\rr d)$ 
(respectively $f \in \Sigma _t^{s}(\rr d)$) if and only if 
$$
|\partial^\alpha f(x)| \lesssim \ep^{|\alpha|} (\alpha!)^s e^{-h|x|^{\frac 1t}}
$$
for some $h>0, \ep>0$ (respectively for every $h>0, \ep>0$). 
Moreover we recall that for $s <1$ the elements of $\mathcal{S}_t^s(\rr d)$ admit entire 
extensions to $\mathbf{C}^d$ satisfying suitable exponential bounds, cf. \cite{GS} for details.

\medspace

The \emph{Gelfand-Shilov distribution spaces} $(\mathcal S_t^{s})'(\rr d)$
and $(\Sigma _t^s)'(\rr d)$ are the projective and inductive limit
respectively of $(\mathcal S_{t,h}^s)'(\rr d)$.  This means that
\begin{equation}\tag*{(\ref{GSspacecond1})$'$}
(\mathcal S_t^s)'(\rr d) = \bigcap _{h>0}(\mathcal S_{t,h}^s)'(\rr d)\quad
\text{and}\quad (\Sigma _t^s)'(\rr d) =\bigcup _{h>0}(\mathcal S_{t,h}^s)'(\rr d).
\end{equation}
We remark that in \cite{Pi2} it is proved that $(\mathcal S_t^s)'(\rr d)$
is the dual of $\mathcal S_t^s(\rr d)$, and $(\Sigma _t^s)'(\rr d)$
is the dual of $\Sigma _t^s(\rr d)$ (also in topological sense).

\par

The Gelfand-Shilov spaces possess several convenient mapping
properties. For example they are invariant under
translations, dilations, and to some extent tensor products
and (partial) Fourier transformations. For conveniency we set
$\maclS _s (\rr d)=\maclS _s^s (\rr d)$ and $\Sigma _s (\rr d)=\Sigma _s^s (\rr d)$, and
similarly for their distribution spaces.

\par

The Fourier transform $\mathscr F$ is the linear and continuous map on $\mascS (\rr d)$,
given by the formula
$$
(\mathscr Ff)(\xi )= \widehat f(\xi ) \equiv (2\pi )^{-\frac d2}\int _{\rr
{d}} f(x)e^{-i\scal  x\xi }\, dx
$$
when $f\in \mascS (\rr d)$. Here $\scal \cdo \cdo$ denotes the usual
scalar product on $\rr d$. 
The Fourier transform extends  uniquely to homeomorphisms
on $\mathcal S_s'(\rr d)$ and on $\Sigma _s'(\rr d)$. Furthermore,
it restricts to homeomorphisms on
$\mathcal S_s(\rr d)$ and on
$\Sigma _s(\rr d)$.

\medspace

Some considerations later on involve a broader family of
Gelfand-Shilov spaces. More precisely, for any $s_j,t_j>0$, $j=1,2$, the Gelfand-Shilov
space $\maclS _{t_1,t_2}^{s_1,s_2}(\rr {2d})$ consists of all
$F\in C^\infty (\rr {2d})$ such that
\begin{equation}\label{GSExtCond}
|x_1^{\beta _1}x_2^{\beta _2}\partial _{x_1}^{\alpha _1}
\partial _{x_2}^{\alpha _2}F(x_1,x_2)| \lesssim
h^{|\alpha _1+\alpha _2+\beta _1+\beta _2|}
(\alpha _1!)^{s_1}(\alpha _2!)^{s_2}(\beta _1!)^{t_1}(\beta _2!)^{t_2}
\end{equation}
for some $h>0$, with topology defined in analogous way as above. In the same way,
$\Sigma _{t_1,t_2}^{s_1,s_2}(\rr {2d})$ consists of all
$F\in C^\infty (\rr {2d})$ such that \eqref{GSExtCond} holds for every $h>0$.

\par

The following proposition explains mapping properties of partial Fourier transforms on
Gelfand-Shilov spaces, and follows by similar arguments as in analogous situations in
\cite{GS}. The proof is therefore omitted. Here, $\mascF _1F$ and $\mascF _2F$ are the partial
Fourier transforms of $F(x_1,x_2)$ with respect to $x_1\in \rr d$ and $x_2\in \rr d$,
respectively.

\par

\begin{prop}\label{propBroadGSSpaceChar}
Let $s_1,s_2,t_1,t_2>0$. Then the following is true:
\begin{enumerate}
\item the mappings $\mascF _1$ and $\mascF _2$ on $\mascS (\rr {2d})$
restrict to homeomorphisms
$$
\mascF _1 \, : \, \maclS _{t_1,t_2}^{s_1,s_2}(\rr {2d}) \to \maclS _{s_1,t_2}^{t_1,s_2}(\rr {2d})
\quad \text{and}\quad
\mascF _2 \, : \, \maclS _{t_1,t_2}^{s_1,s_2}(\rr {2d}) \to \maclS _{t_1,2_2}^{s_1,t_2}(\rr {2d})
\text ;
$$

\vrum

\item the mappings $\mascF _1$ and $\mascF _2$ on $\mascS (\rr {2d})$
are uniquely extendable to homeomorphisms
$$
\mascF _1 \, : \, (\maclS _{t_1,t_2}^{s_1,s_2})'(\rr {2d}) \to (\maclS _{s_1,t_2}^{t_1,s_2})'(\rr {2d})
\quad \text{and}\quad
\mascF _2 \, : \, (\maclS _{t_1,t_2}^{s_1,s_2})'(\rr {2d}) \to (\maclS _{t_1,2_2}^{s_1,t_2})'(\rr {2d}).
$$
\end{enumerate}

\par

The same holds true if the $\maclS _{t_1,t_2}^{s_1,s_2}$-spaces and their duals are replaced by
corresponding $\Sigma _{t_1,t_2}^{s_1,s_2}$-spaces and their duals.
\end{prop}

\par

The next two results follow from \cite{ChuChuKim}. The proofs are therefore omitted.

\begin{prop}
Let $s_1,s_2,t_1,t_2>0$ be such that $s_1+t_1\ge 1$ and $s_2+t_2\ge 1$. Then the following
conditions are equivalent.
\begin{enumerate}
\item $F\in \maclS _{t_1,t_2}^{s_1,s_2}(\rr {2d})$;

\vrum

\item $\displaystyle{|F(x_1,x_2)|\lesssim e^{-h(|x_1|^{\frac 1{t_1}} + |x_2|^{\frac 1{t_2}} )}}$
and
$\displaystyle{|\widehat F(\xi _1,\xi _2)|\lesssim e^{-h(|\xi _1|^{\frac 1{s_1}} + |\xi _2|^{\frac 1{s_2}} )}}$,
for some $h>0$.
\end{enumerate}
\end{prop}

\par

\begin{prop}
Let $s_1,s_2,t_1,t_2>0$ be such that $s_1+t_1\ge 1$ and $s_2+t_2\ge 1$ and $s_j,t_j\neq \frac 12$,
$j=1,2$. Then the following conditions are equivalent.
\begin{enumerate}
\item $F\in \Sigma _{t_1,t_2}^{s_1,s_2}(\rr {2d})$;

\vrum

\item $\displaystyle{|F(x_1,x_2)|\lesssim e^{-h(|x_1|^{\frac 1{t_1}} + |x_2|^{\frac 1{t_2}} )}}$
and
$\displaystyle{|\widehat F(\xi _1,\xi _2)|\lesssim e^{-h(|\xi _1|^{\frac 1{s_1}} + |\xi _2|^{\frac 1{s_2}} )}}$,
for every $h>0$.
\end{enumerate}
\end{prop}

\par

\subsection{The short time Fourier transform and modulation spaces}
Before recalling the definition of modulation spaces, we recall some basic facts about
the short-time Fourier transform and weights.

\par

Let $\phi \in \maclS _{\frac 12}(\rr d)\setminus 0$ be fixed. Then the short-time
Fourier transform of $f\in \maclS _{\frac 12}'(\rr d)$ is given by
$$
(V_\phi f)(x,\xi ) = (2\pi )^{-\frac d2}(f,\phi (\cdo -x)e^{i\scal \cdo \xi})_{L^2}.
$$
Here $(\cdo ,\cdo )_{L^2}$ is the unique extension of the $L^2$-form on
$\maclS _{\frac 12}(\rr d)$ to a continuous sesqui-linear form on $\maclS
_{\frac 12}'(\rr d)\times \maclS _{\frac 12}(\rr d)$. In the case
$f\in L^p(\rr d)$, for some $p\in [1,\infty]$, then $V_\phi f$ is given by
$$
V_\phi f(x,\xi ) \equiv (2\pi )^{-\frac d2}\int _{\rr d}f(y)\overline{\phi (y-x)}
e^{-i\scal y\xi}\, dy .
$$

\par

A function $\omega$ on $\rr d$ is called a \emph{weight}, or \emph{weight function},
if $\omega ,1/\omega \in L^\infty _{\loc} (\rr d)$
are positive everywhere. It is often assumed that $\omega$ is $v$-moderate
for some positive function $v$ on $\rr d$. This means that
\begin{equation}\label{vModerate}
\omega (x+y)\lesssim \omega (x)v(y),\quad x,y\in \rr d.
\end{equation}
If $v$ is even and satisfies \eqref{vModerate} with $\omega =v$, then $v$ is called submultiplicative.

\par

Let $p,q\in [1,\infty ]$, $\omega$ be a weight on $\rr {2d}$, and let
$\phi (x)=\pi ^{-d/4}e^{-|x|^2/2}$. Then the modulation space $M^{p,q}
_{(\omega )}(\rr d)$ consists of all $f\in \maclS _{\frac 12}'(\rr d)$ such that
\begin{equation}\label{Modnorm}
\nm f{M^{p,q}_{(\omega )}} \equiv \left (  \int _{\rr d} \left ( \int _{\rr d}
|V_\phi f(x,\xi )\omega (x,\xi )|^p\, dx\right )^{q/p}\, d\xi \right )^{1/q} <\infty 
\end{equation}
(with obvious modifications when $p=\infty$ or $q=\infty$). We put
$M^p_{(\omega )}=M^{p,p}_{(\omega )}$.

\par

If $e^{-\ep (|x|^2+|\xi |^2)}\lesssim \omega (x,\xi )$ for every $\ep >0$, then
$M^{p,q}_{(\omega )}(\rr d)$ is a Banach space with norm \eqref{Modnorm}
(see \cite{To15}). Furthermore, if $\omega$ is $v$-moderate
for some $v$, then it is $v$-moderate for some new $v$ which is submultiplicative.
In this case,
$$
\Sigma _1(\rr d)\subseteq M^{p,q}_{(\omega )}(\rr d)\subseteq \Sigma _1'(\rr d),
$$
$M^{p,q}_{(\omega )}(\rr d)$ is invariant under the choice of the window function
$\phi \in M^1_{(v)}\setminus 0$, and different choices of $\phi \in M^1_{(v)}\setminus 0$
give rise to equivalent norms \eqref{Modnorm}. (See \cite{Gro}.)

\par

\subsection{Pseudo-differential operators}
Let $t\in \mathbf R$ be fixed and let $a\in \maclS _{\frac 12} (\rr {2d})$. Then the
\emph{pseudo-differential operator} $\op _t(a)$
with \emph{symbol} $a$ is the continuous operator on $\maclS _{\frac 12} (\rr d)$,
defined by the formula
\begin{equation}\label{e0.5}
(\op _t(a)f)(x)
=
(2\pi  ) ^{-d}\iint a((1-t)x+ty,\xi )f(y)e^{i\scal {x-y}\xi }\,
dyd\xi .
\end{equation}

\par

More generally, if $A$ is a real $d\times d$-matrix, then the
\emph{pseudo-differential operator} $\op _A(a)$
with symbol $a$ is the continuous operator on $\maclS _{\frac 12} (\rr d)$,
defined by the formula
\begin{equation}\tag*{(\ref{e0.5})$'$}
(\op _A(a)f)(x)
=
(2\pi  ) ^{-d}\iint a(x-A(x-y),\xi )f(y)e^{i\scal {x-y}\xi }\,
dyd\xi .
\end{equation}
We note that $\op _t(a)=\op _A(a)$ when $A=t\cdot I$ and $I$ is the identity matrix.
The definition of $\op _A(a)$ extends uniquely to any $a\in \maclS _{\frac 12} '(\rr
{2d})$, and then $\op _A(a)$ is continuous from $\maclS _{\frac 12} (\rr d)$ to
$\maclS _{\frac 12} '(\rr d)$ (cf. e.{\,}g. \cite {Ho,To11}). More precisely, if $a\in
\maclS _{\frac 12}'(\rr {2d})$, then $\op _A(a)$ is defined as the continuous
operator from $\maclS _{\frac 12}(\rr d)$ to $\maclS _{\frac 12}'(\rr d)$ with
the kernel
\begin{equation}\label{KaADef}
K_{A,a}(x,y) \equiv (\mascF _2^{-1}a)(x-A(x-y),x-y).
\end{equation}
It is easily seen that the latter definition agrees with \eqref{e0.5}$'$
when $a\in L^1(\rr {2d})$.

\par

If $t=\frac 12$, then $\op _t(a)$ is equal to the Weyl
operator $\op ^w(a)$ for $a$. If instead $t=0$, then the standard
(Kohn-Nirenberg) representation $a(x,D)$ is obtained.

\par

\subsection{Symbol classes}

\par

Next we introduce function spaces related to symbol classes
of the pseudo-differential operators. These functions should obey various
conditions of the form
\begin{equation}
\label{symbols}
|\partial^\alpha f(x)| \lesssim \varepsilon^{|\alpha|}\alpha !^s e^{h |x|^{\frac 1s}}.
\end{equation}

\par

\begin{defn}
Let $s >0$. 
\begin{enumerate}
\item The set $\Gamma^\infty_s (\rr d)$ consists of all $f \in C^\infty(\rr d)$
such that for some $\varepsilon>0$, \eqref{symbols} holds for every $h > 0$.

\vrum

\item 
The set $\Gamma^\infty_{0,s} (\rr d)$ consists of all $f \in C^\infty(\rr d)$
such that for some $h>0$, \eqref{symbols} holds for every $\varepsilon >0$.

\vrum

\item
The set $\Gamma^\infty_{1,s} (\rr d)$ consists of all $f \in C^\infty(\rr d)$
such that \eqref{symbols} holds for some $\ep , h>0$.
\end{enumerate}
\end{defn}

\par

Next we introduce some convenient weight classes.

\par

\begin{defn}\label{P}
Let $s>0.$ Then $\mathscr P_s(\rr d)$ consists of all $\omega \in L^{\infty}
_{\textrm{loc}}(\rr d)$ such that $\omega >0$ and 
\begin{equation} \label{estomega}
\omega (x+y) \lesssim \omega (x)e^{c|y|^{\frac 1s}}, \quad  x,y \in \rr d,
\end{equation}
for some positive constant $c$.
\end{defn}

\par

\begin{rem}\label{remP}
Assume that \eqref{estomega} is true for some $s<1$. Then by
\cite[Lemma 4.2]{Gro2}, \eqref{estomega} is true also when $s=1$,
for some other choice of the constant $c>0$, if necessary.
Furthermore, since $e^{c|y|}\lesssim e^{h|y|^{\frac 1s}}$ for every
$h>0$, it follows that \eqref{estomega} is true for every $c>0$ in
this case. In particular,
$$
e^{-h |x|^{\frac 1s}}\lesssim e^{-c|x|} \lesssim \omega (x)
\lesssim e^{c|x|}\lesssim e^{h|x|^{\frac 1s}}
$$
for some $c>0$ and every $h>0$.
\end{rem}

\par

In Section \ref{sec3} we shall investigate algebraic properties for
pseudo-differential operators  with
symbols in the classes given in the following definition.

\par

\begin{defn}
Let $s>0$ and let $\omega \in \mathscr
P_{s}(\rr {2d})$. The set $S_s^{(\omega)}(\rr {2d})$ consists of all
$a \in C^{\infty}(\rr {2d})$ such that
\begin{equation}\label{SsCond}
|\partial^{\alpha}a(x,\xi)|  \lesssim \omega(x,\xi) \ep^{|\alpha|}(\alpha!)^s,
\end{equation}
for every $\ep >0.$
\end{defn} 

\par

\begin{rem}\label{rem:GammaToSRel}
Let $s>0$ and $\omega \in \mascP _s(\rr d)$. By Remark \ref{remP}
it follows that $S^{(\omega )}_s(\rr d)$ is contained in $\Gamma ^\infty _{0,s}(\rr d)$.
Furthermore, if in addition $s<1$, then it follows from the same remark
that $S^{(\omega )}_s(\rr d)$ is
contained in $\Gamma ^\infty _{s}(\rr d)$, as well.
\end{rem}

\par

Later on we also need the Propositions \ref{LimitSpaces} and \ref{TraceProp} below. The first
proposition deals with convenient sequences which converge to a given element
in $\Gamma ^\infty _{s}(\rr {d})$. Here the sequence $f_\ep$, $\ep >0$, in
$\Gamma ^\infty _{s}(\rr {d})$
is said to converge to $f\in \Gamma ^\infty _{s}(\rr {d})$ as $\ep \to 0^+$, if for some
$h>0$ we have $\nm {f-f_\ep}{s,h,r}\to 0$ as $\ep \to  0^+$, for every $r>0$. Here
$$
\nm f{s,h,r}\equiv \sup _{\alpha \in \nn d}\frac {\nm {e^{-r|\cdo |^{\frac 1s}}\partial ^\alpha f}
{L^\infty}}{\alpha !^sh^{|\alpha |}} ,
$$
and for conveniency we let $\Gamma ^{\infty ,h} _{s}(\rr {d})$ be the set of all
$f\in C^\infty (\rr d)$
such that $\nm f{s,h,r}<\infty$ for every $r>0$.

\par

\begin{prop}\label{LimitSpaces}
Let $s\ge \frac 12$, $h>0$, $f \in \Gamma ^\infty _{s}(\rr {d})$ and
$f_{\ep}=e^{-\ep |\cdo |^{2}}f$, $\ep >0$.
Then $f_\ep \to f$ in $\Gamma ^\infty _{s}(\rr {d})$.
\end{prop}

\par

\begin{proof}
Let
$\phi _\ep = e^{-\ep |\cdo |^2}$. Then $f_\ep =\phi _\ep f$, and
there are constants $C,c,h_0>0$ which are independent of $\ep >0$ such that
$$
|\partial ^\alpha \phi _\ep (x)|
\le Ch_0^{|\alpha |}\ep ^{\frac {|\alpha |}2}\alpha !^{\frac 12}e^{-c\ep |x|^2}.
$$
For conveniency we also let
$$
\nm f{s,h,r ,\alpha } \equiv
\frac {\nm {e^{-r |\cdo | ^{\frac 1s}} \partial ^\alpha f}{L^\infty}}
{\alpha !^sh^{|\alpha |}}.
$$
Then $\nm f{s,h,r}= \sup _\alpha \nm f{s,h,r ,\alpha }$.

\par

By Leibniz rule we get
$$
\partial ^\alpha f_\ep = \sum _{\gamma \le \alpha} {\alpha \choose \gamma}
\partial ^\gamma \phi _\ep \cdot \partial ^{\alpha -\gamma}f.
$$
Hence, if $\ep <r\le 1$, $s=\frac 12$ and $h_1>0$, then
\begin{equation*}
\nm {f_\ep -f}{s,h_1,r ,\alpha} \le J_{1}(s,h_1,r, \ep )+J_{2}(s,h_1,\ep ,\alpha),
\end{equation*}
where
\begin{align*}
J_{1}(s,h_1,r, \ep ) &= \sup _{\alpha \in \nn d}
\frac 
{\nm {e^{-r|\cdo |^2}(1-e^{-\ep |\cdo |^2})\partial ^\alpha f}{L^\infty}}
{\alpha !^{s}h_1^{|\alpha |}}
\intertext{and}
J_{2}(s,h_1,\ep ,\alpha) &= (\alpha !^{s}h_1^{|\alpha |})^{-1}
\sum _{0\neq  \gamma \le \alpha} {\alpha \choose \gamma}
\nm {\partial ^\gamma \phi _\ep}{L^\infty} \nm {\partial ^{\alpha -\gamma}f}{L^\infty} .
\end{align*}

\par

Evidently, since $f\in  \Gamma _{s}^{\infty ,h}(\rr d)$, it follows by
straight-forward
estimates that $J_{1}(s,h_1,r, \ep )\to 0$ as
$\ep \to 0^+$. For $J_{2}(s,h_1,\ep ,\alpha)$ we have
\begin{multline*}
J_{2}(s,h_1,\ep ,\alpha) \lesssim (\alpha !^{s}h_1^{|\alpha |})^{-1}
\sum _{0\neq  \gamma \le \alpha} {\alpha \choose \gamma}
h_0^{|\gamma |}\ep ^{\frac {|\gamma |}2}\gamma !^{s}
h^{|\alpha -\gamma |}(\alpha -\gamma )!^{s}
\\[1ex]
\le
\frac {\ep ^{\frac 12}}{h_1^{|\alpha |}} \sum _{\gamma \le \alpha}
{\alpha \choose \gamma} h_0^{|\gamma |}h^{|\alpha -\gamma |}
= \ep ^{\frac 12}\left ( \frac {h_0+h}{h_1}\right )^{|\alpha |} \le \ep ^{\frac 12},
\end{multline*}
provided $h_1$ is chosen larger than $h_0+h$. This gives the result.
\end{proof}

\par

The next result concerns mapping properties of $\Gamma ^\infty _s$
spaces under trace operators. Let $V_1$ and $V_2$ be vector
spaces such that
\begin{equation}\label{Eq:DirectSum}
V_1,V_2\subseteq \rr d
\quad \text{and}\quad
V_1\oplus V_2 =\rr d,
\end{equation}
and let $x_0\in V_2$ be fixed. Then the trace operator
$T_{\overline V,x_0}$ from $C^\infty (\rr d)$ to $C^\infty (V_1)$
is defined by the formulae
$$
(T_{\overline V,x_0}f)(y) = f(y,x_0),\quad
y\in V_1,\quad \overline V=(V_1,V_2).
$$

\par

\begin{prop}\label{TraceProp}
Let $\overline V=(V_1,V_2)$ and $x_0\in V_2$, where
$V_1$ and $V_2$ are as in \eqref{Eq:DirectSum}. Then
$T_{\overline V,x_0}$ restricts to a continuous mapping
from $\Gamma _s^\infty (\rr d)$ to $\Gamma _s^\infty (V_1)$.
\end{prop}

\par

Proposition \ref{TraceProp} is a straight-forward consequence
of the definitions. The verifications are left for the reader.

\section{The short-time Fourier transform and
regularity}\label{sec2}

\par

In this section we deduce equivalences between conditions on
the short-time Fourier transforms of functions or distributions and estimates
on derivatives. 

\par

In what follows we let $\kappa$ be defined as
\begin{equation} \label{kappadef}
\kappa (r) =
\begin{cases}
1\quad &\text{when}\ r\le 1,
\\[1ex]
2^{r-1}\quad &\text{when}\ r> 1.
\end{cases}
\end{equation}

\par

In the sequel we shall frequently use the well known inequality
$$
|x+y|^{\frac 1s} \leq \kappa(s^{-1}) (|x|^{\frac 1s} + |y|^{\frac 1s}), \qquad s >0,\quad
x,y\in \rr d.
$$

\par

\begin{prop}\label{prop1}
Let $\phi \in \Sigma _s(\rr d)\setminus 0$, $s>\frac 12$, $h\in \mathbf R$ and let
$f\in \cS '_{\frac 12}(\rr d)$. Then the following is true:
\begin{enumerate}
\item If $f\in C^\infty (\rr d)$ and satisfies
\begin{equation}\label{GelfRelCond1}
|\partial ^\alpha f(x)|\lesssim e^{h|x|^{\frac 1s}}\ep ^{|\alpha |}(\alpha !)^s,
\end{equation}
for every $\ep >0$ (resp. for some $\ep >0$), then
\begin{equation}\label{stftcond1}
|V_\phi f(x,\xi )|\lesssim e^{\kappa (s^{-1})h|x|^{\frac 1s}-\ep |\xi |^{\frac 1s}},
\end{equation}
for every $\ep >0$ (resp. for some new $\ep >0$);

\vrum

\item If 
\begin{equation}\label{GelfRelCond1A}
|V_\phi f(x,\xi )|\lesssim e^{h|x|^{\frac 1s}-\ep |\xi |^{\frac 1s}},
\end{equation}
for every $\ep >0$ (resp. for some $\ep >0$), then $f\in C^\infty (\rr d)$ and satisfies
$$
|\partial ^\alpha f(x)|\lesssim e^{\kappa (s^{-1})h|x|^{\frac 1s}}\ep ^{|\alpha |}(\alpha !)^s,
$$
for every $\ep >0$ (resp. for some new $\ep >0$).
\end{enumerate}
\end{prop}

\par


\par

\begin{proof}
We only prove the assertion when \eqref{GelfRelCond1} or \eqref{GelfRelCond1A}
are true for every $\ep >0$, leaving the straight-forward modifications of the
other cases to the reader. 

\par

Assume that \eqref{GelfRelCond1} holds. Then for every $x \in \rr d$ the function
$$
y\mapsto F_x(y)\equiv f(y+x)\overline{\phi (y)}
$$
belongs to $\Sigma _s$, and
$$
|\partial _y^\alpha F_x(y)| \lesssim e^{\kappa (s^{-1})h|x|^{\frac 1s}}e^{-h_0|y|^{\frac 1s}}
\ep ^{|\alpha |}(\alpha !)^s,
$$
for every $\ep , h_0 >0$. In particular,
\begin{equation}\label{FxEsts}
|F_x(y)|\lesssim e^{\kappa (s^{-1})h|x|^{\frac 1s}}e^{-h_0|y|^{\frac 1s}}\quad \text{and}\quad
|\widehat F_x(\xi )|\lesssim e^{\kappa (s^{-1})h|x|^{\frac 1s}}e^{-h_0|\xi |^{\frac 1s}},
\end{equation}
for every $h_0>0$. Since $|V_\phi f(x,\xi )| = |\widehat F_x(\xi )|$, the estimate
\eqref{stftcond1} follows from the second inequality in \eqref{FxEsts}. This
proves (1).

\par

Next we prove (2). By the inversion formula we get
$$
f(x) = (2\pi)^{-\frac d2} \nm \phi{L^2}^{-2} \iint _{\rr {2d}} V_\phi f(y,\eta )
\phi (x-y)e^{i\scal x\eta}\, dyd\eta .
$$
By differentiation and the fact that $\phi \in \Sigma _s$ we get
\begin{multline*}
|\partial ^\alpha f(x)| \asymp \left |
\sum _{\beta \le \alpha} {\alpha \choose \beta} i^{|\beta|}
\iint _{\rr {2d}} \eta ^\beta V_\phi f(y,\eta )  (\partial ^{\alpha -\beta }\phi )(x-y)
e^{i\scal x\eta}\, dyd\eta
\right |
\\[1ex]
\le
\sum _{\beta \le \alpha} {\alpha \choose \beta}
\iint _{\rr {2d}} |\eta ^\beta V_\phi f(y,\eta )  (\partial ^{\alpha -\beta }\phi )(x-y)|
\, dyd\eta
\\[1ex]
\lesssim
\sum _{\beta \le \alpha} {\alpha \choose \beta}
\iint _{\rr {2d}} |\eta ^\beta e^{-\ep |\eta |^{\frac 1s}}e^{h |y|^{\frac 1s}}  (\partial ^{\alpha -\beta }\phi )(x-y)|
\, dyd\eta
\\[1ex]
\lesssim
\sum _{\beta \le \alpha} {\alpha \choose \beta}
\ep _2 ^{|\alpha - \beta |} (\alpha -\beta )!^s
\iint _{\rr {2d}} |\eta ^\beta |e^{-\ep |\eta |^{\frac 1s}}e^{h |y|^{\frac 1s}}
e^{-\ep _1|x-y|^{\frac 1s}} \, dyd\eta ,
\end{multline*}
for every $\ep _1>0$ and $\ep _2>0$. Since
$$
|\eta ^\beta e^{-\ep |\eta |^{\frac 1s}}|\lesssim \ep _2^{|\beta|}(\beta !)^se^{-\ep |\eta |^{\frac 1s}/2},
$$
we get
\begin{multline*}
|\partial ^\alpha f(x)|
\\[1ex]
\lesssim \ep _2 ^{|\alpha |}
\sum _{\beta \le \alpha} {\alpha \choose \beta}
(\beta !
(\alpha -\beta )!)^s \iint _{\rr {2d}} e^{-\ep |\eta |^{\frac 1s}/2} e^{h |y|^{\frac 1s}}
e^{-\ep _1|x-y|^{\frac 1s}} \, dyd\eta
\\[1ex]
\lesssim (2^{1-s}\ep _2)^{|\alpha |}(\alpha !)^s \int _{\rr n} e^{h |y|^{\frac 1s}}
e^{-\ep _1|x-y|^{\frac 1s}} \, dy
\end{multline*}

\par

Since $|y|^{\frac 1s}\le \kappa (s^{-1})(|x|^{\frac 1s}+|y-x|^{\frac 1s})$ and $\ep _1$ can be chosen
arbitrarily large, it follows from the last estimate that
$$
|\partial ^\alpha f(x)| \lesssim (2\ep _2)^{|\alpha |}(\alpha !)^s
e^{h \kappa (s^{-1})|x|^{\frac 1s}},
$$
for every $\ep _2>0$. This gives the result.
\end{proof}

\par

As a consequence of the previous result we get the following.

\par

\begin{prop}\label{prop2}
Let $\phi \in \Sigma_s(\rr d)\setminus 0$, $s > \frac 12$ and let
$f\in \cS'_{\frac 12}(\rr d)$. Then the following conditions are equivalent:
\begin{enumerate}
\item $f\in \Gamma_{0,s}^\infty (\rr d)$ (resp. $f\in \Gamma_{1,s}^\infty (\rr d)$);

\vrum

\item there exists a constant $h>0$ such that
$$
|V_\phi f(x,\xi )|\lesssim e^{h |x|^{\frac 1s}-\ep |\xi |^{\frac 1s}},
$$
for every $\ep >0$ (resp. for some $\ep >0$).
\end{enumerate}
\end{prop}

\par

By similar arguments we also get the following.

\par

\begin{prop}\label{prop2'}
Let $\phi \in \cS _s(\rr d)\setminus 0$, $s\ge \frac 12$ and let
$f\in \cS '_{\frac 12}(\rr d)$. Then the following conditions are equivalent:
\begin{enumerate}
\item $f\in \Gamma_s^\infty (\rr d)$;

\vrum

\item there exists a constant $h>0$ such that
$$
|V_\phi f(x,\xi )|\lesssim e^{\ep |x|^{\frac 1s}-h |\xi |^{\frac 1s}},
$$
for every $\ep >0$.
\end{enumerate}
\end{prop}

\par

\section{Continuity and composition properties for
pseudo-differential operators}\label{sec3}

\par

In this section we deduce continuity, invariance and composition properties for
pseudo-differential operators with symbols in the classes considered in the
previous sections. In the first part we show that for any such class $S$, the
set $\op _A (S)$ of pseudo-differential operators is independent of the 
matrix $A$. Thereafter we deduce that such operators are continuous
on Gelfand-Shilov spaces and their duals. Finally we show that these
operator classes are closed under compositions.

\par

\subsection{Invariance properties}

\par

An important ingredient in these considerations concerns mapping properties
for the operator $e^{i\scal {AD_\xi}{D_x}}$. In fact we have the following.

\par

\begin{thm}\label{ThmCalculiTransf}
Let $s\ge \frac 12$, $\omega \in \mascP _s(\rr {2d})$, $s_1,s_2,t_1,t_2>0$ be such that
$$
s_1+t_1\ge 1,\quad s_2+t_2\ge 1,\quad s_1\le s_2
\quad \text{and}\quad
t_2\le t_1,
$$
and let $A$ be a real $d\times d$ matrix. Then the following is true:
\begin{enumerate}
\item $e^{i\scal {AD_\xi}{D_x}}$ on $\mascS (\rr {2d})$ restricts to
a homeomorphism on $\maclS _{t_1,s_2}^{s_1,t_2}(\rr {2d})$,
and extends uniquely to a homeomorphism on
$(\maclS _{t_1,s_2}^{s_1,t_2})'(\rr {2d})$;

\vrum

\item if in addition $(s_1,t_1)\neq (\frac 12 ,\frac 12)$ and
$(s_2,t_2)\neq (\frac 12 ,\frac 12)$, then $e^{i\scal {AD_\xi}{D_x}}$
on $\mascS (\rr {2d})$ restricts to a homeomorphism on
$\Sigma _{t_1,s_2}^{s_1,t_2}(\rr {2d})$,
and extends uniquely to a homeomorphism on
$(\Sigma _{t_1,s_2}^{s_1,t_2})'(\rr {2d})$;

\vrum

\item $e^{i\scal {AD_\xi}{D_x}}$ is a homeomorphism on $\Gamma _s^\infty (\rr {2d})$.
If in addition $s>\frac 12$, then $e^{i\scal {AD_\xi}{D_x}}$ is a homeomorphism on
$\Gamma _{0,s}^\infty (\rr {2d})$;

\vrum

\item $e^{i\scal {AD_\xi}{D_x}}$ is a homeomorphism on
$S^{(\omega )}_s (\rr {2d})$.
\end{enumerate}
\end{thm}

\par

\begin{cor}
Let $s,t>0$ be such that $s\le t$. Then $e^{i\scal {AD_\xi}{D_x}}$ is
a homeomorphism on $\maclS _t^s(\rr {2d})$, $\Sigma _t^s(\rr {2d})$,
$(\Sigma _t^s)'(\rr {2d})$ and on $(\maclS _t^s)'(\rr {2d})$.
\end{cor}

\par

The assertion (1) in the previous theorem is essentially
a special case of Theorem 32 in \cite{Tr}. In order to be self-contained
we present a complete but different proof.

\par

We need some preparations for the proof and start with the following
proposition.

\par

\begin{prop} \label{prepalg}
Let $s >\frac 12, \phi$ be a Gaussian on $\rr {2d}$ and let $\omega \in
\mascP _s(\rr {2d})$. Then the following conditions are equivalent:
\begin{enumerate} 
\item $ a \in S_s^{(\omega)}(\rr {2d})$;

\vrum

\item $a\in \Sigma _1'(\rr {2d})$ and
\begin{equation}\label{STFTEst1}
\left | \partial _X^\alpha \left ( e^{i\langle X,\Xi  \rangle}V_{\phi}a(X, \Xi) \right )
\right | \lesssim \omega (X) \ep^{|\alpha|}(\alpha!)^s e^{-R|\Xi |^{\frac 1s}},
\end{equation}
for every $\alpha \in \nn d$, $\ep >0$, $R>0$, and $X,\Xi \in \rr {2d}$;

\vrum

\item $a\in \Sigma _1'(\rr {2d})$ and \eqref{STFTEst1} holds for
$\alpha =0$, for every $R>0.$
\end{enumerate}
\end{prop}

\par

\begin{proof}
Obviously, (2) implies (3). Assume now that (1) holds. Let $X=(x,\xi),
Y=(y, \eta)$ and
$$
F_a(X,Y)=a(X+Y)\phi(Y).
$$ 
By straightforward application of Leibniz rule in combination with
\eqref{estomega} we obtain 
$$
|\partial_X^\alpha F_a(X,Y)| \lesssim \omega(X)e^{-R|Y|^{\frac 1s}}
\ep^{|\alpha |}(\alpha !)^s
$$
for every $\ep >0$ and $ R>0$.
Hence, if 
$$
G_{a, \ep, X}(Y)= \frac{(\partial _X^\alpha F_a)(X,Y)}{\omega(X)
\ep^{|\alpha |}(\alpha !)^s},
$$
then $\{ G_{a, \ep, X} \}_{X \in \mathbf{R}^{2d}}$ is a bounded set in
$\Sigma_s(\mathbf{R}^{d})$ for every fixed $\ep >0$. If $\mascF _2 F_a$
is the partial Fourier transform of $F_a(X,Y)$ with respect to the $Y$-variable, we get
$$
|\partial_X^\alpha (\mascF _2 F_a)(X, \zeta, z)| \lesssim \omega(X)e^{-R(|z|^{\frac 1s}
+|\zeta|^{\frac 1s})}\ep^{|\alpha|}(\alpha!)^s,
$$
for every $R>0,\ep>0$. This is the same as (2).

\par

It remains to prove that (3) implies (1), but this follows by similar arguments as in the
proof of Proposition \ref{prop1}.
The details are left for the reader.
\end{proof}

\par

\begin{prop}\label{SymbClassModSpace}
Let $q\in [1,\infty ]$, $s>\frac 12$, $\phi \in \Sigma _s(\rr {2d})\setminus 0$,
$\omega \in \mascP _s(\mathbf{R}^{2d}),$ and let
$$
\omega _R(x,\xi, \eta, y) = \omega(x,\xi) e^{-R(|y|^{\frac 1s} + |\eta|^{\frac 1s})}.
$$
Then 
\begin{multline}\label{iden}
S^{(\omega)}_s(\mathbf{R}^{2d})= \bigcap_{R>0} M^\infty_{(1/\omega _R)}
(\mathbf{R}^{2d}) = \bigcap_{R>0} M^{\infty,1}_{(1/\omega_R)}(\mathbf{R}^{2d})
\\[1ex]
=
\bigcap _{R>0}\sets {a\in \maclS _{\frac 12}'(\rr d)}{\nm {\omega
_R^{-1}V_\phi a}{L^{\infty ,q}} <\infty }.
\end{multline}
\end{prop}

\par

\begin{proof}
The first equality in \eqref{iden} follows immediately from Proposition \ref{prepalg}.
We need to prove that $S^{(\omega)}_s(\mathbf{R}^{2d})$ is equal to
the last intersection in \eqref{iden}. Let $\phi _0\in \Sigma _s(\rr {2d})\setminus 0$, 
$a \in \Sigma'_{s}(\mathbf{R}^{2d})$, and set
\begin{gather*}
F_{0,a}(X,Y)=|(V_{\phi _0}a)(x,\xi, \eta, y)|,
\quad
F_{a}(X,Y)=|(V_{\phi}a)(x,\xi, \eta, y)|
\\[1ex]
\text{and}\quad 
G(X,Y)=|(V_\phi {\phi _0})(x,\xi, \eta, y)|,
\end{gather*}
where $X=(x,\xi)$ and 
$Y=(y,\eta)$. Since $V_\phi {\phi _0}\in \Sigma _s(\rr {4d})$, we have
\begin{equation}\label{GEst1}
0\le G (X,Y) \lesssim e^{-R(|X|^{\frac 1s}+|Y|^{\frac 1s})}
\quad \text{for every}\quad
R>0 .
\end{equation}

\par

By \cite[Lemma 11.3.3]{Gro}, we have
$F_a \lesssim F_{0,a} \ast G.$
Since $s >\frac 12$, we obtain 
\begin{multline}\label{pippo}
(\omega_R^{-1} \cdot F_a)(X,Y)
\\
\lesssim \omega(X)^{-1}e^{R|Y|^{\frac 1s}}
\iint F_{0,a}(X-X_1, Y-Y_1)G(X_1, Y_1)\, dX_1 dY_1 
\\
\lesssim \iint (\omega^{-1}_{cR} \cdot F_{0,a})(X-X_1, Y-Y_1)G_1(X_1,Y_1)\,  dX_1 dY_1
\end{multline}
for some $G_1$ satisfying \eqref{GEst1} and
some $c>0$.
By applying the $L^\infty$-norm on the last inequality we get
\begin{multline*}
\nm {\omega _R^{-1}F_a}{L^{\infty}} \lesssim \sup_{Y}\iint 
\big ( \sup (\omega^{-1}_{cR} \cdot F_{0,a})(\cdot , Y-Y_1)\big )
G_1(X_1, Y_1)\,  dX_1 dY_1
\\
\le
\sup _Y \big ( \nm {(\omega^{-1}_{cR} \cdot F_{0,a})(\cdot -(0,Y)}{L^{\infty ,q}}
\nm {G_1}{L^{1,q'}}
\asymp \nm {\omega^{-1}_{cR} \cdot F_{0,a}}{L^{\infty ,q}}
\end{multline*}

\par

We only consider the case $q<\infty$ when proving the opposite inequality.
The case $q=\infty$ follows by similar arguments and is left for the reader.

\par

By \eqref{pippo} we have 
\begin{multline*}
\| \omega^{-1}_{R} \cdot F_{a} \| _{L^{\infty,q}}^q
\\[1ex]
\lesssim  \int \sup _X
\left( \iint (\omega^{-1}_{cR} \cdot F_{0,a})(X-X_1, Y-Y_1)G(X_1, Y_1)
\, dX_1 dY_1 \right)^q\, dY
\\[1ex]
\lesssim  \int
\left( \iint \sup (\omega^{-1}_{2cR} \cdot F_{0,a})(\cdo , Y-Y_1)
e^{-cR|Y-Y_1|^{\frac 1s}}
G(X_1, Y_1)\, dX_1 dY_1 \right)^q\, dY
\\[1ex]
\lesssim \| \omega^{-1}_{2cR} \cdot F_{0,a} \|_{L^\infty}^q
\int \left( \iint e^{-cR|Y-Y_1|^{\frac 1s}}
G(X_1, Y_1)\, dX_1 dY_1 \right)^q\, dY
\\[1ex]
\asymp \| \omega^{-1}_{2cR} \cdot F_{0,a} \|_{L^\infty}^q.
\end{multline*}
By interchanging the roles of $\phi$ and $\phi _0$ we get
$$
\| \omega^{-1}_{R} \cdot F_{0,a} \| _{L^{\infty,q}} \lesssim 
\| \omega^{-1}_{2cR} \cdot F_{a} \|_{L^\infty},
$$
and the result follows.
\end{proof}

\par

\begin{rem}\label{remSymbClassModSpace}
Let
$$
\omega _{h,\ep ,s}(x,\xi ) \equiv e^{h|x|^{\frac 1s}-\ep |\xi |^{\frac 1s}},
$$
when $h,s, \ep >0$. If $s>\frac 12$, then
\begin{align}
\Gamma ^\infty _{0,s}(\rr d)
&=
\bigcup _{h>0} \left (\bigcap _{\ep >0} M^\infty _{(1/\omega
_{h,\ep ,s})}(\rr d)\right )
=
\bigcup _{h>0} \left (\bigcap _{\ep >0} M^{\infty ,1} _{(1/\omega
_{h,\ep ,s})}(\rr d)\right ).
\label{Gamma0sMod}
\intertext{If instead $s\ge \frac 12$, then}
\Gamma ^\infty _{s}(\rr d)
&=
\bigcup _{h>0} \left (\bigcap _{\ep >0} M^\infty _{(1/\omega _{\ep ,h,s})}(\rr d)\right )
=
\bigcup _{h>0} \left (\bigcap _{\ep >0} M^{\infty ,1} _{(1/\omega _{\ep ,h,s})}(\rr d)\right ).
\label{GammasMod}
\end{align}
In fact, the first equalities in \eqref{Gamma0sMod} and in \eqref{GammasMod}
follow from Propositions \ref{prop2} and \ref{prop2'}. The last equalities follow by
similar arguments as in the previous proof.
\end{rem}

\par

\begin{proof}[Proof of Theorem \ref{ThmCalculiTransf}]
(1) Let $a\in \mascS (\rr {2d})$.
By straight-forward computations we get
$$
\big (\mascF _2^{-1}(e^{-i \scal {AD_\xi}{D_x}} a)\big )(x-Ay,y) = 
(\mascF _2^{-1}a)(x,y).
$$
This implies that $e^{-i\scal {AD_\xi}{D_x}} = \mascF _2 \circ
U_A \circ \mascF _2^{-1}$,
where $U_A$ is the map given by $(U_A F)(x,y)=F(x+Ay,y)$.

\par

We need to show that $U_A$ is continuous on $\maclS _{t_1,t_2}
^{s_1,s_2}(\rr {2d})$. Let
$G=U_AF$, where $F\in \maclS _{t_1,t_2}^{s_1,s_2}(\rr {2d})$. Then
$$
G(x,y)= F(x+Ay,y)
\quad \text{and}\quad
\widehat G(\xi ,\eta )= \widehat F(\xi ,\eta -A^*\eta ),
$$
where $A^*$ denotes the transpose of $A$.
By Proposition \ref{propBroadGSSpaceChar}  and the assumptions on
$s_j$ and $t_j$, there are constants $c_1,c_2,r>0$, depending on $A$,
$s$ and $t$ only such that
\begin{multline*}
|G(x,y)| = |F(x+Ay,y)| \lesssim e^{-r(|x+Ay|^{\frac 1{t_1}}+|y|^{\frac 1{t_2} } ) }
\lesssim 
e^{-c_1r|x|^{\frac 1{t_1}}-c_2r|y|^{\frac 1{t_2}}},
\end{multline*}
and
\begin{equation*}
|\widehat G(\xi ,\eta )| = |\widehat F(\xi ,\eta -A^*\xi )|
\lesssim e^{-r(|\xi |^{\frac 1{s_1}}+|\eta -A^*\xi | ^{\frac 1{s_2}})}
\lesssim 
e^{-c_1r|\xi |^{\frac 1{s_1}}-c_2r|\eta |^{\frac 1{s_2}}}.
\end{equation*}
The result now follows from these estimates and the fact that the topology in
$\maclS _{t_1,t_2}^{s_1,s_2}(\rr {2d})$ can be defined through such estimates.

\par

The assertion (2) follows by similar arguments as for the proof of (1) and
is left for the reader.

\par

Next we prove (4). First we consider the case $s>\frac 12$.
Let $\phi \in \Sigma _s(\rr {2d})$ and $\phi _A = e^{i\scal {AD_\xi }{D_x}}\phi$. Then
$\phi _A\in \Sigma _s(\rr {2d})$, in view of (2).

\par

Also let
$$
\omega _{A,R}(x,\xi ,\eta ,y) = \omega (x-Ay,\xi -A^*\eta )e^{-R(|y|^{\frac 1s}+|\eta |^{\frac 1s})}.
$$
By straight-forward applications of Parseval's formula, we get
$$
|(V_{\phi _A} (e^{i\scal {AD_\xi}{D_x}}a))(x,\xi ,\eta ,y)|
=
|(V_\phi a)(x-Ay,\xi -A^*\eta ,\eta ,y)|
$$
(cf. Proposition 1.7 in \cite{To7} and its proof). This gives
$$
\nm {\omega _{0,R}^{-1}V_{\phi} a}{L^{p,q}} =
\nm {\omega _{A,R}^{-1}V_{\phi _A} (e^{i\scal {AD_\xi }{D_x}}a)}{L^{p,q}}.
$$
Hence Proposition \ref{SymbClassModSpace} gives
\begin{alignat*}{1}
a\in S^{(\omega )} _s(\rr {2d})
\quad &\Leftrightarrow \quad
\nm {\omega _{0,R}^{-1}V_\phi a}{L^\infty}<\infty \ \forall \ R>0
\\[1ex]
&\Leftrightarrow \quad
\nm {\omega _{t,R}^{-1}V_{\phi _A} (e^{i\scal {AD_\xi}{D_x}}a)}{L^\infty} <\infty \ \forall \ R>0
\\[1ex]
&\Leftrightarrow \quad
\nm {\omega _{0,R}^{-1}V_{\phi _A} (e^{i\scal {AD_\xi}{D_x}}a)}{L^\infty} <\infty \ \forall \ R>0
\\[1ex]
&\Leftrightarrow \quad
e^{i\scal {AD_\xi}{D_x}}a\in S^{(\omega )} _s(\rr {2d}), 
\end{alignat*}
and the result follows in this case. Here the third equivalence follows from the fact that
$$
\omega _{0,R+c}\lesssim \omega _{t,R}\lesssim \omega _{0,R-c},
$$
for some $c>0$.

\par

It remains to consider the case $s=\frac 12$. Then $S^{(\omega )}_s(\rr {2d})$ is
contained in the set of
real analytic functions on $\rr {2d}$ which are extendable to entire functions. Hence, if
$a\in S^{(\omega )}_s(\rr {2d})$ and $b=e^{i\scal {AD_\xi}{D_x}}$, then
$$
b = \sum _{k=0}^\infty  \frac{i^{k}}{k!} 
(\scal {AD_\xi}{D_x}^k a).
$$
Now let $\ep >0$ and let $h_0>0$ be chosen as the maximum value of the modulus of the
matrix elements in $A$. Since the number of terms in $\scal {AD_\xi}{D_x}$
is at most $d^2$, it follows from the latter equality and
\eqref{SsCond} that
\begin{multline*}
|D_x^\delta D_\xi ^\gamma b(x,\xi )|
\le
\sum _{k=0}^\infty  \frac{(d^2h_0)^{k}}{k!}
\max _{|\alpha |=|\beta |=k} 
|(D_x ^{\beta +\delta}D_\xi ^{\alpha +\gamma} a)(x,\xi)|
\\[1ex]
\lesssim
\omega (x,\xi ) \sum _{k=0}^\infty  \frac{(d^2h_0)^k}{k!} 
\ep ^{2|\beta |+|\gamma +\delta |}
\max _{|\alpha |=|\beta |=k}
\big ( (\alpha +\gamma )! (\beta +\delta )! \big )^{\frac 12}
\\[1ex]
\le
\omega (x,\xi ) (\gamma !\delta !)^{\frac 12} (\sqrt 2 \ep )^{|\gamma +\delta |}
\sum _{k=0}^\infty  \frac{(2d^2h_0\ep ^2)^k}{k!} 
\max _{|\alpha |=|\beta |=k}
\big ( \alpha ! \beta ! \big )^{\frac 12}
\\[1ex]
\le
\omega (x,\xi ) (\gamma !\delta !)^{\frac 12} (\sqrt 2 \ep )^{|\gamma +\delta |}
\sum _{k=0}^\infty  (2d^2h_0\ep ^2)^k
\\[1ex]
\lesssim \omega (x,\xi ) (\gamma !\delta !)^{\frac 12} (\sqrt 2 \ep )^{|\gamma +\delta |},
\end{multline*}
provided $\ep$ is chosen smaller than $(2d^2h_0)^{-\frac 12}$. Here the third inequality follows from
$$
(\beta +\delta )!\le 2^{|\beta +\delta |}\beta ! \delta ! \, .
$$
Hence $b\in  S^{(\omega )}_s(\rr {2d})$. In the same way it follows that
$a\in S^{(\omega )}_s(\rr {2d})$ when $b\in  S^{(\omega )}_s(\rr {2d})$, and (4) follows.

\par

Finally, (3) follows by similar arguments as (4), using Remark
\ref{remSymbClassModSpace} instead of Proposition \ref{SymbClassModSpace}.
The details are left for the reader.
\end{proof}

\par 

We note that if $A,B$ are real $d\times d$-matrices, and $a$ belongs to
$\Gamma _{0,s}^\infty (\rr {2d})$ or $\Gamma _{s}^\infty (\rr {2d})$, then
the first part of the previous proof shows that
\begin{equation}\label{EqCalculiTransfer}
\op _A(a) = \op _B(b)\quad \Leftrightarrow \quad 
e^{-i\scal {AD_\xi}{D_x}}a = e^{-i\scal {BD_\xi}{D_x}}b .
\end{equation}

\par

We have now the following.

\par

\begin{thm}\label{ThmCalculiTransf2}
Let $s$, $s_j$, $t_j$ be as in Theorem \ref{ThmCalculiTransf},
let $A$ and $B$ be real $d\times d$ matrices, and let $a$ and $b$
be suitable distributions such that $\op _A(a)=\op _B(b)$. Then the
following is true:
\begin{enumerate}
\item $a\in \maclS _{s_1,t_2}^{t_1,s_2}(\rr {2d})$ if and only if $b\in
\maclS _{s_1,t_2}^{t_1,s_2}(\rr {2d})$, and $a\in (\maclS _{s_1,t_2}^{t_1,s_2})'(\rr {2d})$
if and only if $b\in (\maclS _{s_1,t_2}^{t_1,s_2})'(\rr {2d})$;

\vrum

\item $a\in \Sigma
_{s_1,t_2}^{t_1,s_2}(\rr {2d})$ if and only if $b\in
\Sigma _{s_1,t_2}^{t_1,s_2}(\rr {2d})$, and $a\in (\Sigma _{s_1,t_2}^{t_1,s_2})'(\rr {2d})$
if and only if $b\in (\Sigma _{s_1,t_2}^{t_1,s_2})'(\rr {2d})$;

\vrum

\item $a\in \Gamma _s^\infty (\rr {2d})$ if and only if
$b\in \Gamma _s^\infty (\rr {2d})$, and
$a\in \Gamma _{0,s}^\infty (\rr {2d})$ if and only if $b\in
\Gamma _{0,s}^\infty (\rr {2d})$;

\vrum

\item $a\in S^{(\omega )}_s (\rr {2d})$ if and only if $b\in S^{(\omega )}_s (\rr {2d})$.
\end{enumerate}
\end{thm}

\par

Let $s>\frac 12$ and $\omega$ be as in Definition \ref{P}.
Let $\phi \in \Sigma_s(\rr {2d})$. Since any element in $\mathscr P_s(\rr {2d})$ is
$v$-moderate for some $v$, it follows from Proposition \ref{prop1} that 
\begin{equation}
a \in S_s^{(\omega)}(\rr {2d}) \quad \Leftrightarrow \quad
|V_{\phi}a(x,\xi, \eta, y)| \lesssim
\omega(x,\xi)  e^{-\ep (|y|^{\frac 1s}+|\eta|^{\frac 1s})} 
\end{equation}
for every $\ep >0.$

\subsection{Continuity for pseudo-differential operators}

The first result follows by choosing $p=\infty$ and $q=1$ in Theorem 6.2 in
\cite{To11}.

\par

\begin{prop}\label{p5.4}
Let $s >\frac12$, $\kappa$ be as in \eqref{kappadef}. Let
$p,q\in [1,\infty ]$, $h_j,r_j\in \mathbf R$, $j=1,2$, and 
\begin{equation}\label{omegadefs}
\begin{aligned}
\omega _j(x,\xi ) &= e^{h _j(|x|^{\frac 1s} +|\xi |^{\frac 1s})},\quad j=1,2,
\\[1ex]
\omega (x,\xi ,\eta ,y) &= e^{-r_1(|x|^{\frac 1s} +|\xi |^{\frac 1s})
+ r_2(|y|^{\frac 1s} +|\eta |^{\frac 1s})}
\end{aligned}
\end{equation}
If
\begin{equation}\label{rhocond}
\kappa (s^{-1}) h_2+2^{-\frac 1s}r_1\le r_2\quad \text{and}\quad
\kappa (s^{-1}) (h_2 + r_1) \le h_1,
\end{equation}
and $a\in M^{\infty ,1}_{(\omega )}(\rr
{2d})$, then $\op ^w(a)$ from $\mathcal S_{\frac 12}(\rr d)$ to $\mathcal S_{\frac 12}'(\rr
d)$ extends uniquely to a continuous map from
$M^{p,q}_{(\omega _1)}(\rr d)$ to $M^{p,q}_{(\omega _2)}(\rr d)$.
\end{prop}

\par

\begin{proof}
By Theorem 6.2 in \cite{To11}, it suffices to verify that
\begin{equation} \label{weightcond}
\frac {\omega _2(x-y/2,\xi +\eta /2)}{\omega _1(x+y/2,\xi -\eta /2)}
\lesssim \omega (x,\xi ,\eta ,y),
\end{equation}
and by \eqref{omegadefs} and \eqref{rhocond} the latter inequality follows if we prove
\begin{align*}
h_2|2x-y|^{\frac 1s} -h_1|2x+y|^{\frac 1s} &\le 2^{\frac 1s}(r_2|y|^{\frac 1s}-r_1|x|^{\frac 1s})
\intertext{and}
h_2|2\xi +\eta|^{\frac 1s} -h_1|2\xi -\eta |^{\frac 1s} &\le 2^{\frac 1s}(r_2|\eta |^{\frac 1s}
-r_1|\xi |^{\frac 1s}).
\end{align*}
We only prove the first inequality. The other inequality follows by similar arguments
and is left for the reader.

\par

By \eqref{rhocond} we get
\begin{multline*}
h_2|2x-y|^{\frac 1s} -h_1|2x+y|^{\frac 1s}
\\[1ex]
\le h_2(|2x-y|^{\frac 1s}-\kappa (s^{-1})|2x+y|^{\frac 1s}) - r_1 \kappa(s^{-1})|2x+y|^{\frac 1s}
\\[1ex]
\le
2^{\frac 1s}h_2\kappa (s^{-1})|y|^{\frac 1s} - r_1 \kappa(s^{-1})|2x+y|^{\frac 1s}
\\[1ex]
\le
2^{\frac 1s}h_2\kappa (s^{-1})|y|^{\frac 1s} - 2^{\frac 1s}r_1|x|^{\frac 1s}
+
r_1\kappa (s^{-1})|y|^{\frac 1s}
\\[1ex]
=
2^{\frac 1s}\kappa (s^{-1})(h_2+2^{-\frac 1s}r_1)|y|^{\frac 1s}-2^{\frac 1s}r_1|x|^{\frac 1s}
\le
2^{\frac 1s}(r_2|y|^{\frac 1s}-r_1|x|^{\frac 1s}) .
\end{multline*}
This gives the desired estimates.
\end{proof}

\par

\begin{rem} \label{quantindep}
Proposition \ref{p5.4} is valid when $\op^w (a)$ is replaced by $\op_A(a)$
when $A$ is a fixed real $d\times d$  matrix, and 
the condition \eqref{weightcond} is replaced by 
$$
\frac{\omega_2(x-Ay, \xi+(I-A)\eta)}{\omega_1(x+(I-A)y, \xi-A\eta)}
\lesssim \omega(x,\xi,\eta,y).
$$
Here and in what follows we let $I$ be the $d\times d$ unit matrix.
This leads to that the conclusion in Proposition \ref{p5.4} is still valid when
$\op^w (a)$ is replaced by $\op_A(a)$ provided the conditions on the
constants in \eqref{omegadefs} and \eqref{rhocond} are modified in
suitable ways.
\end{rem}

\par

In order to get continuity result for pseudo-differential operators on Gelfand-Shilov spaces,
we combine the previous result with the following special case of Theorem
3.9 in \cite{To11}.

\par

\begin{prop}\label{GSandMod}
Let $p,q\in [1,\infty ]$, $s>\frac 12$, and
set
$$
\omega _h(x,\xi )\equiv e^{h(|x|^{\frac 1s}+|\xi |^{\frac 1s})},\quad h>0.
$$
Then
\begin{alignat*}{2}
\bigcup _{h>0}M^{p,q}_{(\omega _h)}(\rr d) &= \mathcal S_s(\rr d), &\qquad
\bigcap _{h>0}M^{p,q}_{(\omega _h)}(\rr d) &= \Sigma _s(\rr d),
\\[1ex]
\bigcap _{h>0}M^{p,q}_{(1/\omega _h)}(\rr d) &= \mathcal S_s'(\rr d), &\qquad
\bigcup _{h>0}M^{p,q}_{(1/\omega _h)}(\rr d) &= \Sigma _s'(\rr d).
\end{alignat*}
\end{prop}

\par

We have now the following.

\par

\begin{thm}\label{theorem1}
Let $A$ be a fixed real $d\times d$  matrix, $s>\frac 12$, and let
$a\in \Gamma_{0,s}^\infty (\rr {2d})$. Then $\op _A(a)$ is continuous on
$\Sigma _s(\rr d)$, and on $\Sigma _s'(\rr d)$.
\end{thm}

\par

\begin{proof}
By Theorem \ref{ThmCalculiTransf} it suffices to consider the Weyl case
$A=\frac 12\cdot I$.
Let
$$
\omega _\ep (x,\xi ,\eta ,y)= e^{-\kappa
(s^{-1})h(|x|^{\frac 1s}+|\xi |^{\frac 1s})+\ep (|y|^{\frac 1s}+|\eta |^{\frac 1s})}.
$$
Then it follows by Proposition \ref{prop1} that $a\in M^{\infty ,1}_{(\omega _\ep)}$
for every $\ep >0$.

\par

Taking $\ep$ sufficiently large we can choose positive
$$
h_2=h_{2,\ep} =\kappa (s^{-1})^{-1}\ep -2^{-\frac 1s}\kappa (s^{-1})h
\quad \text{and}\quad
h_1=h_{1,\ep} = \kappa (s^{-1})(h+h_2).
$$
Then taking $\omega _j=\omega _{j,\ep}$ as in \eqref{omegadefs}, the hypothesis
in Proposition \ref{p5.4} is fulfilled. Hence
$\op ^w(a)$ is continuous from $M^{p,q}_{(\omega _{1,\ep})}$ to 
$M^{p,q}_{(\omega _{2,\ep })}$ for every sufficiently large $\ep$.

\par

The assertion now follows from Proposition \ref{GSandMod}.
\end{proof}

\par

The next result follows in the case $s>\frac 12$ by similar arguments, after
Proposition \ref{prop1} has been replaced by Proposition \ref{prop2}.
In  order to recapture also the case $s=\frac 12$ we adopt a different argument 
used in \cite{Ca1} for similar statements.

\par

\begin{thm}\label{theorem2}
Let $A$ be a fixed real $d\times d$  matrix, $s\geq \frac 12$, and let
$a\in \Gamma_s^\infty (\rr {2d})$. Then $\op_A (a)$ is continuous on $\cS _s(\rr d)$
and on $\cS '_s(\rr d)$.
\end{thm}

\par

\begin{proof}
By Theorem \ref{ThmCalculiTransf} it suffices to consider the case $A=0$,
that is for the operator 
$$
\op _0(a)f(x)= (2\pi )^{-\frac d2} \int_{\rr d} 
a(x,\xi) \widehat{f}(\xi)e^{i\langle x,\xi \rangle} \, d\xi.
$$
The proof can be easily extended to $\textrm{Op}_A(a)$ for general $A$.

\par

First let
$$
m_s(r)=\sum_{j=0}^{\infty}\frac{r^j}{(j!)^{2s}}, \quad r \geq 0,
$$
and set
$$
m_{s, \tau}(x)= m_s(\tau \px^2), \quad \tau>0,\ x \in \rr d.
$$
It follows from \cite{IV} that $m_{s, \tau}(x)$ satisfies
\begin{equation} \label{paola}
C^{-1}e^{(2s-\ep) \eta^{\frac{1}{2s}}\px^{\frac{1}{s}}} \leq m_{s, \tau}(x) \leq
C e^{(2s+\ep)  \eta^{\frac{1}{2s}}\px^{\frac{1}{s}}} ,
\end{equation}
for every $\ep>0$. Moreover we have
$$
\frac{1}{m_{s, \tau}(x)} \sum_{j=0}^{\infty} \frac{\tau ^j}{(j!)^{2s}}(1-\Delta_{\xi})^j
e^{i\langle x,\xi \rangle} = e^{i\langle x,\xi \rangle}.
$$
For fixed $\alpha, \beta \in \nn d,$ we have
\begin{multline*}
(2\pi )^{\frac d2} x^{\alpha}D_x^{\beta} (\textrm{Op}_0(a)f)(x)
\\[1ex]
= x^{\alpha} \sum
_{\gamma \leq \beta} {\beta \choose {\gamma} } \int_{\rr d}
 \xi^{\gamma} D_x^{\beta-\gamma}a(x,\xi) \widehat{f}(\xi ) e^{i \langle x,\xi \rangle} \,  d\xi 
\\[1ex]
= \frac{x^{\alpha}}{m_{s, \tau}(x)} \sum_{\gamma \leq \beta}
\binom{\beta}{\gamma} 
h_{\tau ,\beta ,\gamma}(x),
\end{multline*}
where
$$
h_{\tau ,\beta ,\gamma}(x) = \sum_{j=0}^{\infty} \frac{\tau ^j}{(j!)^{2s}}
 \int_{\rr d}  (1-\Delta_{\xi} )^j
\left [  \xi^{\gamma} D_x^{\beta-\gamma}a(x,\xi) \widehat{f}(\xi ) \right ]  e^{i \langle x,\xi \rangle} \, d\xi.
$$
By \eqref{symbols}, there exist $\varepsilon, h_1, C >0$ such that
\begin{multline*}
|h_{\tau,\beta ,\gamma}(x)| \lesssim e^{h |x|^{\frac{1}{s}}}
\sum_{j=0}^{\infty} (C\tau )^j \ep ^{|\beta - \gamma |} (\beta - \gamma ) !^{s}
\int_{\rr d} |\xi|^{|\gamma |} e^{- (h_1-h) |\xi |^{\frac{1}{s}}}\, d\xi .
\end{multline*}
for every $h>0$.

\par

Hence, taking $h$ sufficiently small, letting $\tau < C^{-1}$, and using
\eqref{paola} and standard estimates for binomial and factorial coefficients,
it follows that there exists $h>0$ depending only on $s$ and $\tau$ only
such that
$$
\sup_{x \in \rr d} \left | x^{\alpha}D_x^{\beta} (\textrm{Op}_0(a)f)(x) \right |
\lesssim h^{|\alpha+\beta|} (\alpha ! \beta!)^{s}.
$$
This concludes the first part of the proof.
The extension to the dual space $\cS '_s(\rr d)$ is straightforward
and is left for the reader.
\end{proof}

\par

\vspace{0.3cm}

The following result follows by similar arguments as in the previous proof. The
verifications are left for the reader.

\par

\begin{thm}\label{theorem3}
Let $A$ be a $d\times d$ real matrix, $s>\frac 12$ and let $a\in \Gamma^\infty_{1,s} (\rr {2d})$.
Then $\op _A(a)$ is continuous from $\Sigma _s(\rr d)$ to $\cS _s(\rr d)$, and
from $\cS '_s(\rr d)$ to $\Sigma _s'(\rr d)$.
\end{thm}

\par

\begin{rem}
Let $A$ be a $d\times d$ real matrix, $s>\frac 12$, $\omega \in \mascP _s(\rr {2d})$,
and let $a\in S^{(\omega )}_s(\rr {2d})$. By Remark \ref{rem:GammaToSRel} and
Theorem \ref{theorem1} it follows that $\op _A(a)$ is continuous on $\Sigma _s(\rr d)$
and on $\Sigma _s'(\rr d)$.

\par

If instead $\frac 12\le s<1$, then Remark \ref{rem:GammaToSRel} and
Theorem \ref{theorem2} show that $\op  _A(a)$ is continuous on
$\cS _s(\rr d)$ and on $\cS _s'(\rr d)$.
\end{rem}

\par

\subsection{Compositions of pseudo-differential operators}

\par

Next we deduce algebraic properties of pseudo-differential operators considered in
Theorem \ref{ThmCalculiTransf}. We recall that for pseudo-differential operators
with symbols in e.{\,}g. H{\"o}rmander classes, we have
$$
\op _0(a_1\wpr _0a_2) = \op _0(a_1)\circ \op _0(a_2),
$$
when
$$
a_1\wpr _0 a_2 (x,\xi ) \equiv \big ( e^{i\scal {D_\xi}{D_y}} (a_1(x,\xi )a_2(y,\eta ) \big )
\Big \vert _{(y,\eta )=(x,\xi )}.
$$
More generally, if $A$ is a real $d\times d$ matrix and $a_1\wpr _Aa_2$ is defined
by
\begin{equation}\label{SharpADef}
a_1\wpr _Aa_2  \equiv e^{i\scal {AD_\xi }{D_x}}\left ( \big ( e^{-i\scal {AD_\xi }{D_x}}a_1 \big )
\wpr _0 \big ( e^{-i\scal {AD_\xi }{D_x}}a_2 \big ) \right ) ,
\end{equation}
for $a_1$ and $a_2$ belonging to certain H{\"o}rmander symbol classes, then
it follows from the analysis in \cite{Ho} that
\begin{equation}\label{SharpAconjugation}
\op _A(a_1\wpr _Aa_2) = (\op _A(a_1)\circ \op _A(a_2)) .
\end{equation}

\par

By Theorem \ref{ThmCalculiTransf}, the definition \eqref{SharpADef} still makes sense for
symbol classes considered in Theorem \ref{ThmCalculiTransf}, and in the following result
we show that \eqref{SharpAconjugation} holds true for such symbol classes.

\par

\begin{thm}\label{Composition}
Let $A$ be a real $d\times d$-matrix, and let $\omega_j
\in \mathscr P_{s}(\rr {2d}), j=1,2$. Then the following is true:
\begin{enumerate}
\item If $s\ge \frac 12$,
$a_j \in S_s^{(\omega_j)}(\rr {2d}), j=1,2$, then $a_1 \wpr _A a_2
\in S_s^{(\omega_1 \omega_2)}(\rr {2d})$, and \eqref{SharpAconjugation}
hold true.

\vrum

\item If $s> \frac 12$,
$a_j \in \Gamma ^\infty _{0,s}(\rr {2d}), j=1,2$, then
$a_1 \wpr _A a_2\in \Gamma ^\infty _{0,s}(\rr {2d})$, and
\eqref{SharpAconjugation} hold true.

\vrum

\item If $s\ge \frac 12$,
$a_j \in \Gamma ^\infty _{s}(\rr {2d}), j=1,2$, then $a_1 \wpr _A a_2\in \Gamma ^\infty _{s}(\rr {2d})$, and
\eqref{SharpAconjugation} hold true.
\end{enumerate}
\end{thm}

\par

For the proof of Theorem \ref{Composition} (3) we have the following.

\par

\begin{prop}\label{LimitsSymbSpaces}
Let $h>0$, $A$ be a real $d\times d$-matrix, $a,b \in \Gamma ^\infty _{s}(\rr {2d})$,
and set $a_{\ep _1,\ep _2}(x,\xi )=e^{-(\ep _1 |x|^2+\ep _2|\xi |^2)}a(x,\xi )$,
$\ep _1,\ep _2>0$.
Then the following is true:
\begin{enumerate}
\item $\displaystyle{a_{\ep _1,\ep_2} \wpr _0 b\to a\wpr _0 b}$ as
$\ep _1,\ep _2\to 0^+$.

\vrum

\item $\displaystyle{\op _0(a_{\ep _1,\ep _2})g \to \op _0(a)g}$
as $\ep _1,\ep _2\to 0^+$, and $g\in \maclS _{\frac 12}(\rr d)$.
\end{enumerate}
\end{prop}

\par

\begin{proof}
The assertion (1) follows from Proposition \ref{LimitSpaces} and its proof,
Proposition \ref{TraceProp} and Theorem \ref{ThmCalculiTransf} (3).

\par

(2) Let $g_\ep$ be defined by $\widehat g_\ep = e^{-\ep |\cdo |^2}\widehat g$, and let
$$
a_\ep (x,\xi )\equiv e^{-\ep |x|^2}a(x,\xi ).
$$
Then $g_\ep \to g$ in $\maclS _{\frac 12}(\rr d)$, and
(1) gives $a_\ep \to a$ and $a_{\ep _1,\ep _2} \to a$ in $\Gamma
_{\frac 12}^{\infty ,h_1}(\rr {2d})$ as $\ep ,\ep _1 ,\ep _2\to 0^+$.

\par

This gives
\begin{multline*}
\lim _{\ep _1\to 0^+} \big ( \lim _{\ep _2\to 0^+} (\op _0(a_{\ep _1,\ep _2})g)  \big )
=
\lim _{\ep _1\to 0^+} \big ( \lim _{\ep _2\to 0^+}
e^{-\ep _1|\cdo |^2}(\op _0(a)g_{\ep _2})  \big )
\\[1ex]
=\op _0(a)g,
\end{multline*}
and (2) follows.
\end{proof}

\par

\begin{proof}[Proof of Theorem \ref{Composition}]
(1) By the previous proposition, it suffices to
consider the Weyl case, $A=\frac 12\cdot I$. Let $\omega _0=\omega _1\omega _2$,
and set
$$
\omega _{j,R}(X,Y) = \omega _j(X)e^{-R|Y|^{\frac 1s}},\quad j=0,1,2 .
$$
We claim
\begin{equation}\label{star}
\frac 1{\omega _{0,R}(X,Y)} \lesssim \frac 1{\omega _{1,R_0}(X-Y+Z,Z)}
\cdot
\frac 1{\omega _{2,R_0}(X+Z,Y-Z)},
\end{equation}
for some $R_0=cR+c_0$, where $c$ and $c_0$ are independent of $R>0$.

\par

In fact,
\begin{multline*}
\frac{1}{\omega_{0,R}(X,Y)}= \frac{e^{R|Y|^{\frac 1s}}}{\omega_0(X)} 
\lesssim \frac{e^{cR|Z|^{\frac 1s}}e^{cR|Y-Z|^{\frac 1s}}}{\omega_1(X)\omega_2(X)} 
\\[1ex]
\lesssim \frac{e^{cR|Z|^{\frac 1s}}e^{c_0|Y-Z|^{\frac 1s}}}{\omega_1(X-Y+Z)} \cdot
\frac{e^{cR|Y-Z|^{\frac 1s}}e^{c_0|Z|^{\frac 1s}}}{\omega_2(X+Z)}
\\[1ex]
= \frac{1}{\omega_{1,R_0}(X-Y+Z,Z)} \cdot \frac{1}{\omega_{2,R_0}(X+Z,Y-Z)},
\end{multline*}  
for some $R_0=cR+c_0$, where $c, c_0$ are independent of $R$. 

\par

Now assume that $a_j \in S_s^{(\omega_j)}, j=1,2.$ Since $S_s^{(\omega_j)}$
is contained in $M^{\infty,1}_{(1/\omega_{j,R})}$ for every $R>0$, Theorem 6.4
in \cite{To11} and \eqref{star} show that $a_1 \wpr a_2$ is uniquely defined and
belongs to $M^{\infty,1}_{(1/\omega_{0,R})}$ for every $R>0$. That is, $a_1
\wpr a_2 \in \bigcap\limits_{R>0} M^{\infty,1}_{(1/\omega_{0,R})}
= S_s^{(\omega_0)}$.

\par

The result now follows from the fact that $\op ^w(a_1\wpr a_2) = \op ^w(a_1)\circ
\op ^w(a_2)$ when $a_1\in M^{\infty ,1}_{(1/\omega _{1,R})}(\rr {2d})$ and
$a_2\in \maclS _{\frac 12}(\rr {2d})$, that $\maclS _{\frac 12}(\rr {2d})$ is narrowly dense in
$M^{\infty,1}_{(1/\omega_{0,R})}$, and that $a_1\wpr a_{2,j}\to a_1\wpr a_2$
narrowly as $a_{2,j}\to a_2$ narrowly.

\par

The assertions (2) and (3) in the case $s>\frac 12$ follow by similar
arguments as the proof of (1), and are left for the reader. The assertion
(3) in the case $s=\frac 12$ follows from Proposition \ref{LimitsSymbSpaces}.
\end{proof}

\par

\end{document}